\documentclass[times]{oupau}
\usepackage{latexsym}
\usepackage{amssymb}
\usepackage{amsmath}
\usepackage{amsthm}
\usepackage{enumitem}
\usepackage{xcolor}
\usepackage{cite}
\usepackage{epstopdf}
\RequirePackage{graphicx}
\usepackage{subfigure}

\newtheorem{assumption}[theorem]{Assumption}
\newtheorem{idealization}[theorem]{Idealization}

\newcommand{\R}{\mathbb R}

\newtheorem{algorithm}[theorem]{Algorithm}
\newcommand{\PP}{\mathbb P}
\newcommand{\E}{\mathbb E}

\newcommand{\abs}[1]{\left| #1\right|}

\newcommand{\bracket}[1]{\left[#1\right]}
\newcommand{\set}[1]{\left\{#1\right\}}

\newcommand{\tcorr}{T_{\rm corr}}
\newcommand{\tphase}{T_{\rm phase}}
\newcommand{\tpoll}{T_{\rm poll}}
\newcommand{\tsim}{T_{\rm sim}}
\newcommand{\tpar}{T_{\rm par}}
\newcommand{\refe}{{\rm ref}}
\newcommand{\tacc}{T_{\rm acc}}
\newcommand{\xacc}{X_{\rm acc}}

\newcommand{\algorithmbox}[1]{\fcolorbox{black}[HTML]{E9F0E9}{\parbox{\textwidth}{#1}}}

\begin{document}

\runningheads{D. Aristoff, T. Leli\`evre, and G. Simpson}{ParRep for
  simulating Markov chains}

\title{The parallel replica method for simulating long trajectories of
  Markov chains}

\author{David Aristoff\affil{a}\corrauth, Tony Leli\`evre\affil{b},
  and Gideon Simpson\affil{c}}

\address{\affilnum{a}School of Mathematics, University of Minnesota\\
  \affilnum{b}CERMICS, \'Ecole des Ponts ParisTech\\
  \affilnum{c}Department of Mathematics, Drexel University}
\corraddr{daristof@umn.edu}

\begin{abstract}
  The parallel replica dynamics, originally developed by
  A.F. Voter, efficiently simulates very long trajectories of
  metastable Langevin dynamics. 
  We present an analogous algorithm for discrete time
  Markov processes. Such Markov processes naturally arise, for example,
  from the time discretization of a continuous time stochastic dynamics.
Appealing to properties of quasistationary
  distributions, we show that our algorithm
  reproduces exactly (in some limiting regime) the law of the original trajectory, coarsened over the
  metastable states. 
\end{abstract}

\keywords{Markov chain, parallel computing, parallel replica dynamics,
  quasistationary distributions, metastability}

\received{XXX}

\maketitle

\section{Introduction}
\label{s:intro}

We consider the problem of efficiently simulating time homogeneous
Markov chains with {\it metastable states}: subsets of state 
space in which the Markov chain remains for a long time before leaving. 
By a Markov chain we mean a {\em discrete time} stochastic 
process satisfying the Markov property.
Heuristically, a set $S$ is metastable for a given Markov chain if the Markov chain 
reaches local equilibrium in $S$ much faster than 
it leaves $S$. We will define local equilibrium precisely below, using 
{\em quasistationary distributions} (QSDs).  The simulation of an exit event from
a metastable state using a naive integration technique can be very time consuming.

Metastable Markov chains arise in many contexts. The dynamics of physical systems are often modeled by 
memoryless stochastic processes, including Markov chains, with 
widespread applications in physics, chemistry, and
biology. In computational statistical physics (which is the main
application field we have in mind), 
such models are used to understand macroscopic properties of matter, 
starting from an atomistic description. 
The models can be discrete or continuous in time. 
The discrete in time case has particular importance:  
even when the underlying model is continuous in time, 
what is simulated in practice is a Markov chain obtained by 
time discretization. In the context of computational statistical physics, a widely used continuous time 
model is the Langevin dynamics~\cite{Lelievre:2010uu}, while a 
popular class of discrete time models are the Markov State 
Models~\cite{Prinz:2011id, Chodera:2007bs}. 
For details, see~\cite{Schutte:2013aa,Lelievre:2010uu}. 
For examples of discrete time models 
not obtained from an underlying continuous time dynamics, 
see~\cite{scoppola-94,bovier-2002}. In this article, we propose an
efficient algorithm for simulating metastable Markov chains over very
long time scales. Even though one of our motivations is to treat 
time discretized versions of continuous time models, 
we do not discuss errors in exit events due to time discretization; we refer for
example to~\cite{bouchard-geiss-gobet-2013} 
and references therein for an analysis of this error. 

In the physical applications above, metastability arises from the 
fact that the microscopic time scale (i.e., the physical time between
two steps of the Markov chain) is much smaller than the
macroscopic time scale of interest (i.e., the physical time to observe
a transition between metastable states). Both 
energetic and entropic barriers can contribute to metastability. 
Energetic barriers correspond to high energy saddle 
points between metastable states in the potential energy landscape, while entropic 
barriers are associated with narrow pathways  between metastable states; see Figure~\ref{f:metastable}.

\begin{figure}
  \begin{center}
    \subfigure[Energetic
    Barriers]{\includegraphics[width=8cm]{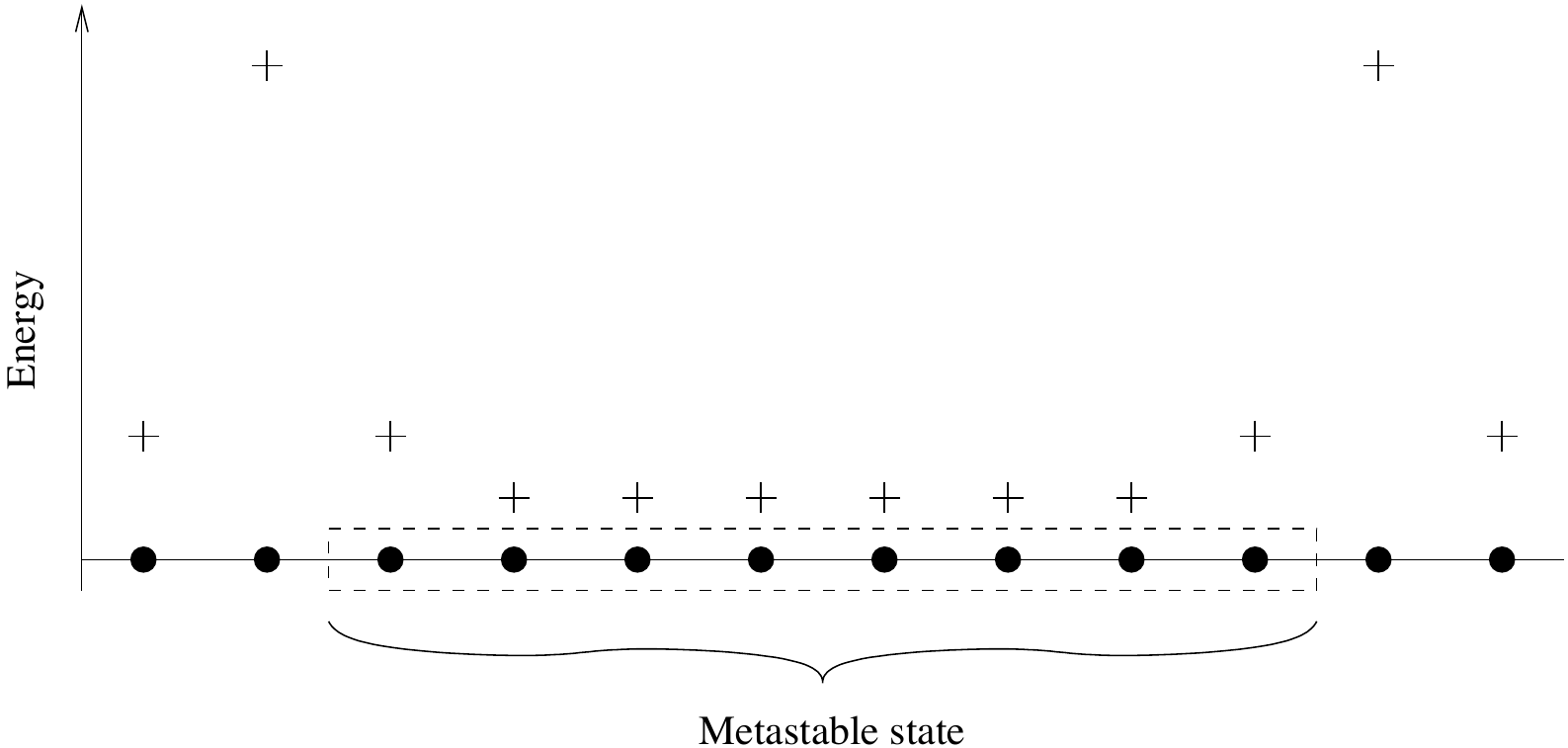}}
    \subfigure[Entropic
    Barriers]{\includegraphics[width=8cm]{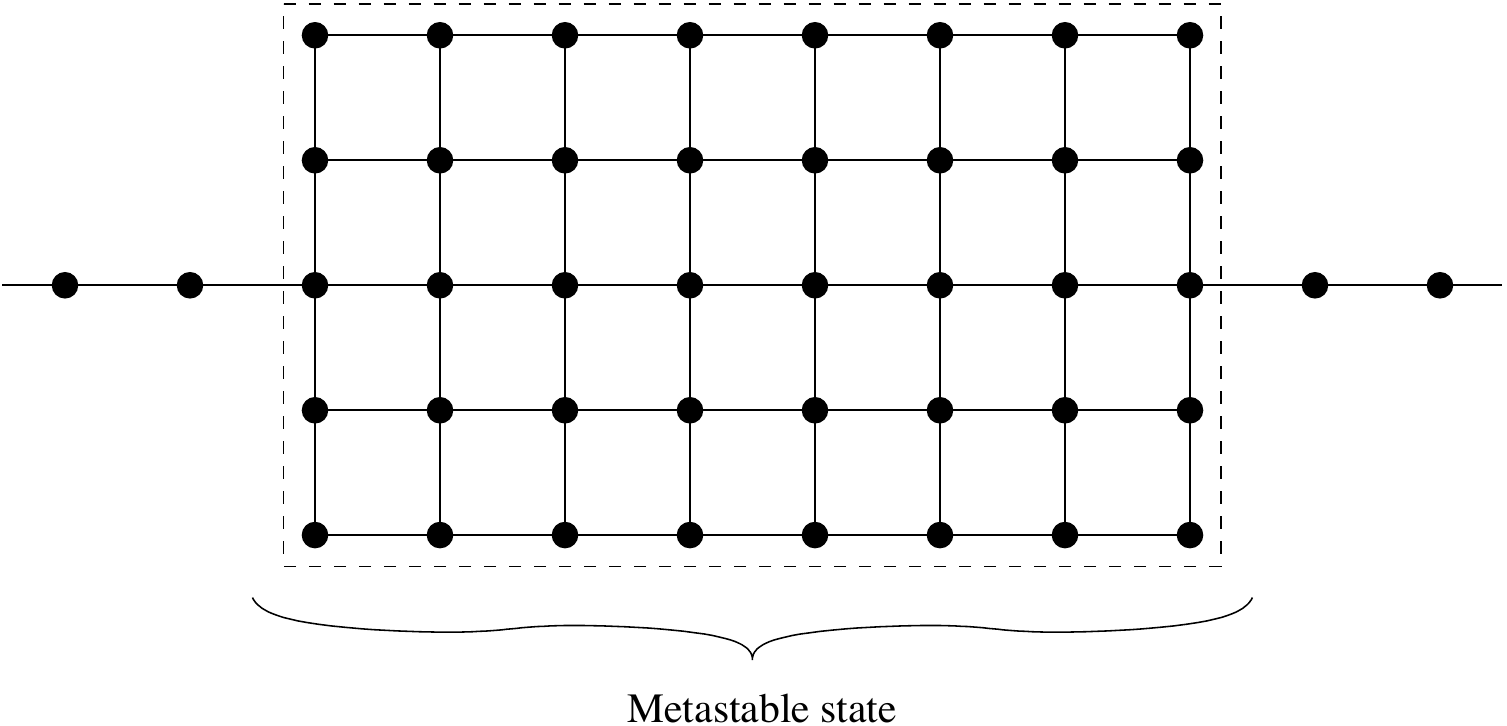}}
  \end{center}
  \caption{(a) Energetic and (b) entropic metastable states of a discrete
    configuration space Markov chain. The chain jumps from one point to another
    according to the following Metropolis dynamics. If $X_n = x$, a direction 
    (in (a), left or right; in (b), up, down, left, or right) is selected uniformly at random. If there is 
    a point $y$ which neighbors $x$ in this direction, then with probability 
    $\min\{1, e^{V(x)-V(y)}\}$ we take $X_{n+1} = y$; otherwise $X_{n+1} = x$. 
    Here, $V$ is a given potential energy function. On the left, each point has only
    two neighbors, and the potential energy is represented on the $y$-axis. On
    the right, each point has the same potential energy and between
    2 and 4 neighbors.}
  \label{f:metastable}
\end{figure}

Many algorithms exist for simulating metastable
stochastic processes over long time scales. One of the most
versatile such algorithms is the  {\em parallel replica
  dynamics} (ParRep) developed by A.F. Voter and co-workers~\cite{Voter,Voter:2002p12678}.
ParRep can be used with both energetic and entropic barriers,   
and it requires no assumptions about temperature, 
barrier heights, or reversibility. The algorithm was
developed to efficiently compute transitions between
metastable states of Langevin dynamics. 
For a mathematical analysis of ParRep in its original
continuous time setting, see~\cite{Gideon,Tony}. 
In this article, we present an algorithm which is an adaptation 
of ParRep to the discrete time setting. It applies to any 
Markov chain.

ParRep uses many replicas of the process, simulated in parallel
asynchronously, to rapidly find transition pathways out of
metastable states. The gain in efficiency over direct simulation comes
from distributing the computational effort across many processors,
parallelizing the problem in time. The cost is that the trajectory
becomes coarse-grained, evolving in the set of metastable states
instead of the original state space. The continuous 
time version of ParRep has been successfully 
used in a number of problems in materials science 
(see e.g.~\cite{Uberuaga:2007ho,Perez:2010dk,Perez:2013ge,Lu:2014hl,Joshi:2013wd,Komanduri:2000df,Baker:2012he}), 
allowing for atomistic resolution while also reaching extended time scales of
microseconds, $10^{-6}$ s. For reference, the microscopic time scale -- 
typically the period of vibration of bond lengths -- is about
$10^{-15}$ s.

In the continuous time case, consistency of the 
algorithm relies on the fact that first exit times from  metastable states are
exponentially distributed. Thus, if $N$ independent
identically distributed (i.i.d.) replicas have 
first exit times $T_i$, $i=1,\ldots, N$, then $N \min(T_1, \ldots,
T_N)$ has the same law as $T_1$. Now if 
$K = \arg\min(T_1, \ldots, T_N)$ is the first replica which leaves the
metastable state
amongst all the replicas, then the simulation clock is advanced
by $N T_K$, and this time agrees in law with the 
original process. In contrast, in the
discrete time case, the exit times from metastable states are
geometrically distributed.  Thus, if $\tau_i$
are now the geometrically distributed first exit times, 
then $N \min(\tau_1,\ldots, \tau_N)$ does not agree in law 
with $\tau_1$.  A different function of the $\tau_i$ must 
be found instead. This is our achievement with Algorithm~\ref{algorithm1} and 
Proposition~\ref{proposition1}. Our algorithm is 
based on the observation that $N[\min(\tau_1,
\ldots, \tau_N)-1] + \min[i \in \{1, \ldots ,N \} , \, \tau_i =  \min(\tau_1,
\ldots, \tau_N) ]$ agrees in law with $\tau_1$.

This article is organized as follows. 
In Section~\ref{QSD}, we formalize the notion of local equilibrium 
using QSDs. 
In Section~\ref{parrep} we present our discrete time 
ParRep algorithm, and in Section~\ref{parrepmath} we study its consistency. 
Examples and a discussion
follow in Section~\ref{EXAMPLE}.

\section{Quasistationary Distributions}\label{QSD}

Throughout this work, $(X_n)_{n\ge 0}$ will be a time homogeneous
Markov chain with values in a probability space $(\Omega,{\mathcal F},\PP)$. 
For a random variable $X$ and probability measure $\mu$, we
write $X\sim \mu$ to indicate $X$ is distributed according to
$\mu$. For random variables $X$ and $Y$, we write $X \sim Y$ when $Y$
is a random variable with the same law as $X$.  
 We write $\mathbb{P}^{\mu}(X_n \in A)$ and 
$\mathbb{E}^{\mu}[f(X_n)]$ to denote probabilities and expectations 
for the Markov chain $(X_n)_{n\ge 0}$ starting from the indicated
initial distribution: $X_0 \sim \mu$.  In the case that
$X_0 = x$, we write ${\mathbb P}^{x}(X_n \in A)$ and
$\mathbb{E}^{x}[f(X_n)]$ to denote probabilities and expectations 
for the Markov chain starting from $x$.

To formulate and apply ParRep, we first need to define the 
metastable subsets of $\Omega$, which 
we will simply call {\em states}. The states will be used to
coarse-grain the dynamics.
\begin{definition}
  Let ${\mathcal S}$ be the collection of states, which we
  assume are disjoint bounded measurable subsets of~$\Omega$. We write
  $S$ for a generic element of ${\mathcal S}$, and $\Pi:\Omega \to
  \Omega/{\mathcal S}$ for the quotient map identifying the states.
\end{definition}
As we will be concerned with when the chain exits  states,
we define the first exit time from $S$,
  % Given a probability measure $\mu$ with support in $S \in {\mathcal
  % S}$, we set
  \begin{equation*}
    \tau := \min\set{n\ge 0\,:\,X_n \notin S}.
  \end{equation*}
Much of the algorithm and analysis depends on the properties of the
QSD, which we now define.
\begin{definition}
  A probability measure $\nu$ with support in $S$ is a QSD if for all
  measurable $A \subset S$ and all $n \in {\mathbb N}$,
  \begin{equation}\label{QSD1}
    \nu(A) = {\mathbb P}^\nu\left(X_n \in A\,|\,\tau > n \right).
  \end{equation}
\end{definition}
Of course both $\tau$ and $\nu$ depend on $S$, but for ease of notation, we do not make 
this explicit. The QSD can be seen as a local equilibrium
reached by the Markov chain, conditioned on the event that it 
remains in the state. Indeed, it is easy to check that if $\nu$ is a
measure with support in $S$ such that,
\begin{equation}\label{QSD2}
  \text{for any measurable $A\subset S$ and any $\mu$ with support in
    $S$},\quad  \nu(A) = \lim_{n\to \infty} {\mathbb P}^\mu\left(X_n \in A\,|\,\tau > n\right),
\end{equation}
then $\nu$ is the QSD, which is then unique. In Section~\ref{QSDmath}, 
we give sufficient conditions for existence and uniqueness of
the QSD and for the convergence~\eqref{QSD2} to occur (see
Theorem~\ref{theorem1}). We refer the reader
to
\cite{Cattiaux:2009um,del2004feynman,Martinez:1994vn,Meleard:2012vl,Tony,collet13:_quasi_station_distr}
for additional properties of the QSD.

\section{The Discrete Time ParRep Algorithm}
\label{parrep}

Using the notation of the previous section, the aim of the ParRep
algorithm is to efficiently generate a trajectory $({\hat X}_n)_{n\ge
  0}$ evolving in $\Omega /{\mathcal S}$ which has, approximately, the
same law as the reference coarse-grained trajectory $(\Pi(X_n))_{n \ge 0}$. 
Two of the parameters in the algorithm -- $\tcorr = \tcorr(S)$ and $\tphase =
\tphase(S)$, called the {\em decorrelation} and {\em dephasing times} --
depend on the current state $S$, but for ease of notation we 
do not indicate this explicitly. See the remarks below Algorithm~\ref{algorithm1}.

% As in the original algorithm, a decorrelation time $\tcorr$ and a
% dephasing time $\tphase$ are associated to each subset $S$ (see the
% remarks below). Again we omit the dependence of $\tcorr$ and
% $\tphase$ on $S$.

\begin{algorithm}
  \label{algorithm1}
  Initialize a reference trajectory $X_0^{\refe} \in \Omega$. Let $N$
  be a fixed number of replicas and $\tpoll$ a fixed polling time at which
  the replicas resynchronize.  Set the simulation clock to zero:
  $\tsim = 0$. A coarse-grained trajectory $({\hat X}_n)_{n\ge 0}$
  evolving in $\Omega /{\mathcal S}$ is obtained by iterating the
  following: \vskip2pt
  \noindent
  \algorithmbox{{\bf Decorrelation Step:} Evolve the reference
    trajectory $(X_n^{\refe})_{n \ge 0}$ until it spends $\tcorr$
    consecutive time steps in some state $S \in {\mathcal S}$. 
    Then proceed to the dephasing
    step. Throughout this step, the simulation clock $\tsim$ is
    running and the coarse-grained trajectory is given by
    \begin{equation}\label{proj1}
      {\hat X}_{\tsim} = \Pi(X_{\tsim}^{\refe}).
    \end{equation}}
  \noindent
  \algorithmbox{ {\bf Dephasing Step:} The simulation clock $\tsim$ is
    now stopped and the reference and coarse-grained trajectories do
    not evolve.  Evolve $N$ independent replicas $\set{X_n^j}_{j=1}^N$
    starting at some initial distribution with support in $S$, such
    that whenever a replica leaves $S$ it is restarted at the initial
    distribution. When a replica spends $\tphase$ consecutive time
    steps in $S$, stop it and store its end position. 
    When all the replicas have stopped, reset
    each replica's clock to $n=0$ and proceed to the parallel step.}
  \noindent
  \algorithmbox{{\bf Parallel Step:} Set $M = 1$ and iterate the
    following:
    \begin{enumerate}
    \item Evolve all $N$ replicas $\set{X_n^j}_{j=1}^N$ from time $n =
      (M-1)\tpoll$ to time $n = M\tpoll$. The simulation clock $\tsim$
      is not advanced in this step.

    \item If none of the replicas leaves $S$ during this time, update
      $M = M+1$ and return to 1, above.

      Otherwise, let $K$ be the smallest number $j$ such that $X_n^j$
      leaves $S$ during this time, let $\tau^K$ be the corresponding
      (first) exit time, and set
      \begin{equation}
        \label{e:exit_update}
        \xacc = X_{\tau^K}^K,\quad \tacc = (N-1)(M-1)\tpoll +
        (K-1)\tpoll + \tau^K.
      \end{equation}
      Update the coarse-grained trajectory by
      \begin{equation}\label{proj2}
        {\hat X}_n = \Pi(S) \quad \hbox{for}\quad n \in [\tsim,
        \tsim+ \tacc-1],
      \end{equation}
      and the simulation clock by $\tsim = \tsim + \tacc$.  Set
      $X_{\tsim}^{\refe} = \xacc$, and return to the decorrelation
      step.
    \end{enumerate}
  }
\end{algorithm}
The idea of the parallel step is to compute the exit time from $S$ as
the sum of the times spent by the replicas up to the first exit
observed among the replicas. More precisely, if we imagine the
replicas being ordered by their indices ($1$ through $N$), this sum is
over all $N$ replicas up to the last polling time, and then over the
first $K$ replicas in the last interval between polling times, $K$
being the smallest index of the replicas which are the first to exit.
Notice that $M$ and $\tau^K$ are such that $\tau^K \in  [(M-1) \tpoll+1,
M \tpoll]$. See Figure~\ref{fig0} for a schematic of the Parallel Step. 
We comment that the formula for updating the simulation time in 
the parallel step of the original ParRep algorithm is simply $\tacc=N \tau^K$.

A few remarks
are in order (see \cite{Gideon,Tony} for additional comments on
the continuous time algorithm):
\begin{description}

\item[The Decorrelation Step.] In this step, the reference trajectory
  is allowed to evolve until it spends a sufficiently long time in a
  single state. At the termination of the decorrelation step, the
  distribution of the reference trajectory should be, according
  to~\eqref{QSD2}, close to that of the QSD (see 
  Theorem~\ref{theorem1} in Section~\ref{QSDmath}).

  The evolution of the reference trajectory is {\em exact} in the
  decorrelation step, and so the coarse-grained trajectory is also
  exact in the decorrelation step.

\item[The Dephasing Step.] The purpose of the dephasing step is to
  generate $N$ i.i.d. samples from the QSD.  While we have described a simple
  rejection sampling algorithm, there is another technique~\cite{Binder:aa} based on a
  branching and interacting particle process sometimes called
  the Fleming-Viot particle process~\cite{ferrari07:_quasi_flemin_viot}.  See
  \cite{Bieniek:2011jf,Bieniek:2012jg,del2004feynman,Grigorescu:2004bs,Meleard:2012vl}
  for studies of this process, and~\cite{Binder:aa} for a discussion of 
  how the Fleming-Viot particle process may be used in ParRep.

  In our rejection sampling we have flexibility on where to initialize
  the replicas.  One could use the position of the reference chain at
  the end of the decorrelation step, or any other point in $S$.

\item[The Decorrelation and Dephasing Times.] $\tcorr$ and $\tphase$
  must be sufficiently large so that the distributions of both the
  reference process and the replicas are as close as possible to the
  QSD, without exhausting computational resources.  $\tphase$ and
  $\tcorr$ play similar roles, and they both depend on the initial
  distribution of the processes in $S$.

  % ; thus it may not be adequate to have them equal. {\bf Can we
  % really buit on that ? With the technique we are currently testing
  % with Gideon, they would actually be equal, no ?}

  Choosing good values of these parameters is nontrivial, as they
  determine the accuracy of the algorithm. In \cite{Binder:aa}, 
  the Fleming-Viot particle process together with convergence 
  diagnostics are used to determine these parameters on the fly 
  in each state. They can also be postulated from some {\it a priori} 
  knowledge (e.g., barrier height between states), if available.

\item[The Polling Time.]

  The purpose of the polling time $\tpoll$ is to permit for periods of
  asynchronous computation of the replicas in a distributed computing
  environment.  For the accelerated time to be correct, it is
  essential that all replicas have run for at least as long as replica
  $K$. Ensuring this requires resynchronization, which occurs 
  at the polling time. 
  
  If communication amongst the replicas is cheap or there is little
  loss of synchronization per time step, one can take $\tpoll
  =1$. In this case, $M=\min\{n\,:\, \exists j \in \{1, \ldots, N\} \,s.t.\,
  X_n^j \not \in S\}$ is the first exit time observed among the $N$
  replicas, $K=\min\{j\,:\, X_M^j \not \in S\}$ (so $M=\tau^K$)
  and $\tacc=N(\tau^K-1)  +K$.

  % Indeed, this can be considered the {\em definition} of a
  % metastable state.

\item[Efficiency of the Algorithm.] For the algorithm to be
  efficient, the states must be truly metastable: within
  each state, the typical time to reach the QSD ($\tcorr$ and
  $\tphase$) should be small relative to the typical exit time. 
  If most states are not metastable, then the exit times
  will be typically smaller than the decorrelation times,  
  and the algorithm will rarely proceed to the dephasing 
  and parallel steps.

  The algorithm is consistent even if some or all the states are not 
   metastable. Indeed, the states can be {\em any} collection of disjoint sets. However, if
  these sets are not reasonably defined, it will be difficult to
  obtain any gain in efficiency with ParRep. Defining the 
  states requires some {\em a priori} knowledge about the system.

\end{description}

% The continuous in time algorithm cannot be directly applied to the
% discrete in time setting because the exit time from a metastable
% region in the continuous in time case is exponentially distributed,
% while it is geometrically distributed in the discrete in time
% case. Thus the usual time correction, based on the stability of the
% exponential law under minimization -- $N \min(T_1, \ldots, T_N)$ has
% the same law as $T_1$, if the $T_i$'s are i.i.d. exponential random
% variables -- does not apply.

\begin{figure}
  \begin{center}
    \includegraphics[width=10cm]{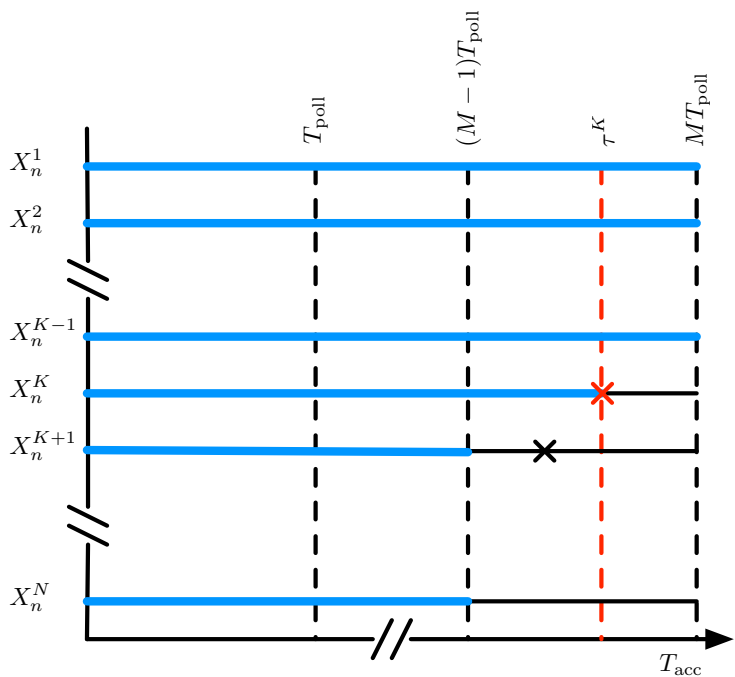}
  \end{center}
  \caption{A schematic of the parallel step.  The horizontal lines
    represent the trajectories of replicas $1,\ldots N$ while the
    crosses correspond to exit events.  Index $K$ is as defined as in
    Algorithm~\ref{algorithm1}.  Here, $M$ cycles internal to the
    parallel step have taken place.  The thicker lines correspond to
    the portions of the chains contributing to~$\tacc$.  }
  \label{fig0}
\end{figure}

\section{Mathematical Analysis of Discrete Time ParRep}
\label{parrepmath}

The main result of this section, Proposition~\ref{proposition1}, shows that
the coarse-grained trajectory simulated in ParRep is {\em exact} if
the QSD has been exactly reached in the decorrelation and dephasing
steps; see Equation~\eqref{eq:theorem3} below.

\subsection{Properties of Quasistationary Distributions}
\label{QSDmath}
Before examining ParRep, we give a condition for existence 
and uniqueness of the QSD. We also state important properties 
of the exit law starting from the QSD. Many of these results 
can be found in~\cite{del2004feynman,collet13:_quasi_station_distr}. 
We assume the following, which is sufficient
 to ensure existence and uniqueness 
of the QSD.
\begin{assumption}
  \label{assumption1}
  Let $S \in {\mathcal S}$ be any state.
  \begin{enumerate}
  \item For any $x\in S$, ${\mathbb P}^x(X_1 \in S) > 0.$
  \item There exists $m\geq 1$ and $\delta \in (0,1)$, such that for
    all $x,y\in S$ and all bounded non-negative measurable functions
    $f:S\to {\mathbb R}$, $ {\mathbb E}^x\left[f(X_m)\,1_{\{\tau >
        m\}}\right] \ge \delta \, {\mathbb E}^y\left[f(X_m)1_{\{\tau >
        m\}}\right].$
    % \begin{equation*}
    % \end{equation*}
  \end{enumerate}
  % and any $x\in S$, we assume
  % \begin{equation*}
  % \end{equation*}
  % Also, there exists $m \ge 1$ and $\delta \in (0,1)$ such that for
  % all $x,y \in S$ and all bounded, non-negative, measurable
  % functions $f:S \to \mathbb R$,
\end{assumption}
With this condition, the following holds (see \cite[Theorem 1]{delmoral}):
\begin{theorem}\label{theorem1}
  Under Assumption~\ref{assumption1}, there exists a unique QSD $\nu$
  in $S$.  Furthermore, for any probability measure $\mu$ with support
  in $S$ and any bounded measurable function $f:S \to {\mathbb R}$,
  \begin{equation}
    \label{e:qsd_convergence}
    \left|{\mathbb E}^{\mu}\left[f(X_n)\,|\,\tau > n\right] 
      -\int_S f(x)\,\nu(dx) \right|
    \le \|f\|_{\infty} \,4 \, \delta^{-1}(1-\delta^2)^{\lfloor n/m \rfloor}.
  \end{equation}
  % where $\lfloor n/m \rfloor$ denotes the integer part of $n/m$.
\end{theorem}
Theorem~\ref{theorem1} shows that the law of $(X_n)_{n\ge 0}$, conditioned on
not exiting $S$, converges in total variation norm to the QSD $\nu$ as $n\to
\infty$.  Thus, at the end of the decorrelation and dephasing steps, if 
$\tcorr$ and $\tphase$ are sufficiently large,  
then the law of the reference process and replicas will be close to that of the
QSD. Notice that Theorem~\ref{theorem1} provides 
an explicit error bound in total variation norm.

% We see in Theorem~\ref{theorem1} that for the distributions of the
% reference process in the decorrelation step and of the replicas in
% the dephasing step to be close to the QSD in total variation norm,
% one may take $\tcorr$ and $\tphase$, in place of $n$, in
% \eqref{e:qsd_convergence} sufficiently large so as to make the right
% hand side small.

Next we state properties of the exit law 
starting from the QSD which are essential to our analysis. While these results are well-known
(see, for instance,
\cite{del2004feynman,collet13:_quasi_station_distr}), we give brief
proofs for completeness.
\begin{theorem}\label{theorem2}
  If $X_0 \sim \nu$, with $\nu$ the QSD in $S$, then $\tau$ and $X_\tau$
  are independent, and $\tau$ is geometrically distributed with
  parameter $p = {\mathbb P}^{\nu} (X_1\notin S)$.
\end{theorem}
\begin{proof}
  Let $k(x,dy)$ denote the transition kernel of $(X_n)_{n\ge 0}$. We compute
  \begin{align*} {\mathbb E}^\nu\left[f(X_{\tau})\,|\,\tau = n\right]
    = \frac{{\mathbb E}^\nu\left[f(X_n)\,1_{\{\tau = n\}}\right]}
    {{\mathbb E}^\nu\left[1_{\{\tau = n\}}\right]} &= \frac{{\mathbb
        E}^\nu\left[1_{\{\tau > n-1\}}\int_{\Omega\setminus S}
        f(y)k(X_{n-1},dy)\right]}
    {{\mathbb E}^\nu\left[1_{\{\tau > n-1\}}\int_{\Omega\setminus S} k(X_{n-1},dy)\right]}\\
    &= \frac{{\mathbb E}^\nu\left[\int_{\Omega\setminus S}
        f(y)k(X_{n-1},dy)\,\big|\,\tau > n-1\right]}
    {{\mathbb E}^\nu\left[\int_{\Omega\setminus S} k(X_{n-1},dy)\,\big|\, \tau > n-1\right]} \\
    &= \frac{\int_S \left(\int_{\Omega\setminus S}
        f(y)k(x,dy)\right)\nu(dx)} {\int_S \left(\int_{\Omega\setminus
          S} k(x,dy)\right)\nu(dx)} ={\mathbb
      E}^\nu\left[f(X_{\tau})\,|\,\tau = 1\right].
  \end{align*}
  The second to last equality is an application of \eqref{QSD1}.  As
  ${\mathbb E}^\nu\left[f(X_{\tau})\,|\,\tau = 1\right]$ is
  independent of $n$, this establishes independence of $\tau$ and $X_\tau$.

  Concerning the distribution of $\tau$, we first calculate
  \begin{equation*} {\mathbb P}^\nu(\tau > n) ={\mathbb
      P}^\nu\left(\tau > n\big|\tau > n-1\right) {\mathbb P}^\nu(\tau
    > n-1)
  \end{equation*}
  and then again use \eqref{QSD1}:
  \begin{align*} {\mathbb P}^\nu\left(\tau > n\big|\tau > n-1\right)=
    \frac{{\mathbb E}^\nu\left[1_{\{\tau > n\}}\right]} { {\mathbb
        P}^\nu(\tau > n-1)} &= \frac{{\mathbb E}^\nu\left[1_{\{\tau >
          n-1\}}\int_S k(X_{n-1},dy)\right]} { {\mathbb
        P}^\nu(\tau > n-1)}\\
    &= {\mathbb E}^\nu\left[\int_S k(X_{n-1},dy)\,\big|\,\tau > n-1\right]\\
    &= \int_S \left(\int_S k(x,dy)\right)\,\nu(dx) = {\mathbb
      P}^{\nu}(X_1 \in S).
  \end{align*}
  Thus, ${\mathbb P}(\tau^\nu > n) = {\mathbb P}(X_1^\nu \in
  S){\mathbb P}(\tau^\nu > n-1)$ and by induction, $ {\mathbb
    P}^\nu(\tau > n) = \left[{\mathbb P}^\nu(X_1 \in S)\right]^n =
  (1-p)^n$.
  % \begin{equation*}
  %  ,
  % \end{equation*}
  % and, by induction,
  % \begin{equation*}
  % .
  % \end{equation*}
\end{proof}

\subsection{Analysis of the exit event}\label{exitevent}

We can now state and prove our main result. 
We make the following idealizing assumption, which
allows us to focus on the the parallel step in
Algorithm~\ref{algorithm1}, neglecting the errors 
due to imperfect sampling of the QSD. 
\begin{idealization}
\label{assumption2}
  Assume that:
  \begin{itemize}
  \item[(A1)] After spending $\tcorr$ consecutive time steps in $S$,
    the process $(X_n)_{n\ge 0}$ is {\em exactly} distributed
    according to the QSD~$\nu$ in $S$. In particular, at the end of
    the decorrelation step, $X_{\tsim}^{\refe} \sim \nu$.
  \item[(A2)] At the end of the dephasing step, all $N$ replicas are
    i.i.d. with law {\em exactly} given by $\nu$.  \end{itemize}
\end{idealization}
Idealization~\ref{assumption2} 
is introduced in view of Theorem~\ref{theorem1}, 
which ensures that the QSD sampling 
error from the dephasing and decorrelation steps 
vanishes as $\tcorr$ and $\tphase$ become large. 
Of course, for finite $\tcorr$ and $\tphase$, 
there is a nonzero error; 
this error will indeed propagate in time, 
but it can be controlled in terms of these two
parameters. For a detailed analysis in the continuous time case, 
see~\cite{Gideon,Tony}. Though the arguments
in~\cite{Gideon,Tony} could be adapted to our time discrete
setting, we do not go in this direction; instead we 
focus on showing consistency of the parallel step.

% , and not to analyze the error.

% To analyze the error due to finite decorrelation and dephasing
% times, one must specify the system under consideration; convergence
% to the QSD is related to the spectral properties of the semigroup
% induced by the process.  For the continuous in time problem, see
% \cite{Gideon,Tony}. The purpose of this work is mainly to adapt the
% parallel step to the discrete in time setting, not to analyze the
% error induced by a particular choice of parameters, so we will not
% consider this error.

% In particular, convergence to the QSD is related to the first two
% eigenvalues of the infinitesimal generator of the process with
% Dirichlet boundary conditions outside of $S$; see for
% example~\cite{Gideon,Tony}.\footote{ {\bf note from gs} It's not the
% second eigenvalue in the discrete }

  Under Idealization~\ref{assumption2}, we show that ParRep is {\em exact}. 
  That is, the
  trajectory generated by ParRep has the same probability law as the
  true coarse-grained chain:
  \begin{equation}~\label{eq:theorem3}
    ({\hat X}_n)_{n\ge 0} \sim (\Pi(X_n))_{n \ge 0}.
  \end{equation}
The evolution of the ParRep coarse-grained trajectory is {exact} in the
decorrelation step. Together with Idealization~\ref{assumption2}, 
this means~\eqref{eq:theorem3} holds if the 
parallel step is consistent (i.e. exact, if all 
replicas start at i.i.d. samples of the QSD). 
This is the content of the following proposition.
% Observe that the evolution of the coarse-grained trajectory is {\em
% exact} in the decorrelation step, and the coarse-grained trajectory
% is not updated during the dephasing step, so the error in
% Algorithm~\ref{algorithm1} can arise only in the parallel step. In
% light of this, (A1) and Theorem~\ref{theorem2},
% Theorem~\ref{theorem3} is a corollary of the following proposition
% which analyzes the exit event.
\begin{proposition}\label{proposition1}
Assume that the $N$ replicas at the beginning of the parallel step are
i.i.d.  with law {\em exactly} given by the QSD $\nu$ in $S$ (this is
Idealization \ref{assumption2}-(A2)). Then the parallel step of
  Algorithm~\ref{algorithm1} is exact:
\begin{equation*}
 (\xacc, \tacc) \sim (X_\tau, \tau),
\end{equation*}
where $(\xacc,\tacc)$ is defined as in Algorithm~\ref{algorithm1}, 
while $(X_\tau, \tau)$ is defined for $(X_n)_{n\ge 0}$ 
starting at $X_0 \sim \nu$. 
\end{proposition}
To prove Proposition~\ref{proposition1}, we need the following lemma:
\begin{lemma}\label{lemma1}
  Let $\tau^{1}, \tau^{2},\ldots, \tau^{N}$ be i.i.d. geometric random
  variables with parameter $p$: for $t \in {\mathbb N}\cup \{0\}$,
  \[ {\mathbb P}(\tau^{j} > t) = (1-p)^{t}.
  \]
  % \begin{equation*}
  % \end{equation*}
  Define
  \begin{align*}
    M &= \min\{m \ge 1\,:\, \exists\, j \in
    \{1,\ldots,N\}\,\,\,s.t.\,\,\,\tau^{j} \le m \tpoll\},\\
    K &= \min\{j \in \{1,\ldots,N\}\,:\, \tau^{j} \le M\tpoll\},\\
    \xi&= (N-1)(M-1)\tpoll + (K-1)\tpoll + \tau^{K}.
  \end{align*}
  Then $\xi$ has the same law as $\tau^{1}$.
\end{lemma}
\begin{proof}
  Notice that $\xi$ can be rewritten as
$$\xi= N(M-1)\tpoll + (K-1)\tpoll + [\tau^{K}-(M-1) \tpoll].$$
Indeed, any natural number $z$ can be uniquely expressed as
$z=N(m-1)\tpoll + (k-1) \tpoll + t$ where $m \in \mathbb N \setminus
\{0\}$, $k \in \{1,\ldots,N\}$ and $t \in \{1,2,\ldots,\tpoll\}$. For
such $m$, $k$ and $t$ we compute
\begin{align*}
  &{\mathbb P}\left(\xi = N(m-1) \tpoll + (k-1) \tpoll + t\right) 
={\mathbb P}\left(M=m,\, K=k, \, \tau^K - (M-1)\tpoll = t\right) \\
  &= {\mathbb P}\left(\tau^{1} > m \tpoll,\, \ldots,\,\tau^{k-1}>m
    \tpoll,\, \tau^{k} = (m-1) \tpoll + t,\,
    \tau^{k+1} > (m-1) \tpoll, \ldots, \tau^{N} > (m-1) \tpoll\right)\\
  &= \mathbb P(\tau^{1} > m \tpoll)^{k-1} {\mathbb P}\left(\tau^{k} =
    (m-1) \tpoll+t\right)
  \left[{\mathbb P}(\tau^{k+1} > (m-1) \tpoll)\right]^{N-k}\\
  &= (1-p)^{(k-1)m\tpoll}p(1-p)^{(m-1)\tpoll+t-1}(1-p)^{(N-k)(m-1)\tpoll}\\
  &= p(1-p)^{N(m-1)\tpoll + (k-1)\tpoll + t-1}= {\mathbb
    P}\left(\tau^{1} = N(m-1)\tpoll + (k-1)\tpoll + t\right).
\end{align*}
% This concludes the proof of Lemma~\ref{lemma1}.
\end{proof}

We can now proceed to the proof of Proposition~\ref{proposition1}.
\begin{proof} 
In light of Theorem~\ref{theorem2}, it suffices to prove:
  \begin{itemize}
  \item[(i)] $\tacc$ is a geometric random variable with parameter $p=
    {\mathbb P}^\nu(X_1 \notin S)$,
 \item[(ii)] $\xacc$ and $X_{\tau}$ have the same law: $\xacc\sim X_{\tau}$, and 
  \item[(iii)] $\tacc$ is independent of $\xacc$, 
  \end{itemize}
  where $(X_n)_{n\ge 0}$ is the process starting 
  at the $X_0 \sim \nu$. 

  We first prove {\em (i)}. For $j \in \{1,2,\ldots,N\}$, let
  $\tau^{j}$ be a random variable representing the first exit time
  from $S$ of the $j$th replica in the parallel step of ParRep, if the
  replica were allowed to keep evolving indefinitely. By (A2),
  $\tau^{1},\ldots,\tau^{N}$ are independent and all have the same
  distribution as $\tau$. 
  Now by Theorem~\ref{theorem2},
  $\tau^{1},\ldots,\tau^{N}$ are i.i.d. geometric random variables
  with parameter $p$, so by Lemma~\ref{lemma1}, $\tacc$ is also a
  geometric random variable with parameter $p$.

  Now we turn to {\em (ii)} and {\em (iii)}. Note that $K = k$ if and
  only if $\xacc = X_{\tau^k}^k$ and there exists $m \in {\mathbb N}$
  such that $\tau^{1} > m\tpoll,\ldots, \tau^{k-1}>m\tpoll$,
  $(m-1)\tpoll <\tau^{k} \le m\tpoll$, and $\tau^{k+1}
  >(m-1)\tpoll,\ldots,\tau^{N}>(m-1)\tpoll$. From
  Theorem~\ref{theorem2} and (A2), $X_{\tau^k}^k$ is independent of
  $\tau^1,\ldots,\tau^N$, so $\xacc$ must be independent
  of $K$. From this and (A2), it follows that $\xacc \sim
  X_{\tau}$. To see that $\xacc$ is independent
  of $\tacc$, let $\sigma(K,\tau^K)$ be the sigma algebra generated by
  $K$ and $\tau^K$. Knowing the value of $K$ and $\tau^K$ is
  enough to deduce the value of $\tacc$; that is, $\tacc$ is
  $\sigma(K,\tau^K)$-measurable. Also, by the preceding
  analysis and Theorem~\ref{theorem2}, $\xacc = X_{\tau^K}^K$ is
  independent of $\sigma(K,\tau^K)$. To conclude that $\tacc$ and
  $\xacc$ are independent, we compute for suitable test
  functions $f$ and $g$:
  \begin{align*}
    {\mathbb E}[f(\tacc)g(\xacc)]&= {\mathbb E}[{\mathbb
      E}[f(\tacc)g(\xacc)\,|\,\sigma(K,\tau^K)]]\\
    &= {\mathbb E}[f(\tacc){\mathbb
      E}[g(\xacc)\,|\,\sigma(K,\tau^K)]]= {\mathbb E}[f(\tacc)]\,
    {\mathbb E}[g(\xacc)].
  \end{align*}
  % This concludes the proof of Proposition~\ref{proposition1}.
\end{proof}

\section{Numerical Examples}\label{EXAMPLE}

In this section we consider two examples. The first illustrates
numerically the fact that the parallel step in Algorithm~\ref{algorithm1} 
is consistent. The second shows typical errors resulting from 
a naive application of the original ParRep algorithm to a 
time discretization of Langevin dynamics. These are simple
illustrative numerical examples. For a more advanced application, we
refer to the paper~\cite{Binder:aa}, where our Algorithm~\ref{algorithm1} 
was used to study the 2D Lennard-Jones cluster of seven atoms. 

\subsection{One-dimensional Random Walk}

Consider a random walk on ${\mathbb Z}$ with transition probabilities
$p(i,j)$ defined as follows:
\begin{equation*}
  p(i,j) =  \begin{cases}
    3/4,& i< 0 \hbox{ and } j=i+1,\\
    1/4, & i< 0 \hbox{ and } j=i-1,\\
    1/3,& i=0 \hbox{ and } |j| \le 1,\\
    1/4,& i> 0 \hbox{ and } j=i+1,\\
    3/4, & i> 0 \hbox{ and } j=i-1,\\
    0, &\hbox{otherwise}.
  \end{cases}
\end{equation*}
We use ParRep to simulate the first exit time $\tau$ of the
random walk from $S=[-5,5]$, starting from the QSD $\nu$ in $S$.
 At
each point except $0$, steps towards $0$ are more likely than
steps towards the boundaries $-5$ or $5$.

We perform this simulation by using the dephasing and parallel steps
of Algorithm~\ref{algorithm1}; for sufficiently large $\tphase$, the accelerated
time $\tacc$ should have the same law as $\tau$. In this simple example 
we can analytically compute the distribution of $\tau$. 
We perform $10^6$ independent ParRep simulations to obtain statistics
on the distribution of $\tacc$ and the gain in ``wall clock time,'' 
defined below.  We
find that $\tacc$ and $\tau$ have very close
probability mass functions when $\tphase = 25$; see Figure~\ref{fig1}. 
To measure the gain in wall clock efficiency
using ParRep, we introduce the {parallel time} $\tpar$ -- defined,
using the notation of Algorithm~\ref{algorithm1}, by $\tpar = M\tpoll$,
where we recall $M$ is such that $\tau^K \in  [(M-1)\tpoll+1,M\tpoll]$.  Thus,
the wall clock time of the parallel step is $C_0 T_{par}$, with $C_0$ 
the computational cost of a single time step of the Markov chain for
one replica.  Note in Figure~\ref{fig2}
the significant parallel time speedup in ParRep compared with the
direct sampling time. The speedup is approximately
linear in~$N$.

\begin{figure}[htbp]
  \begin{center}
    \includegraphics[width=12cm]{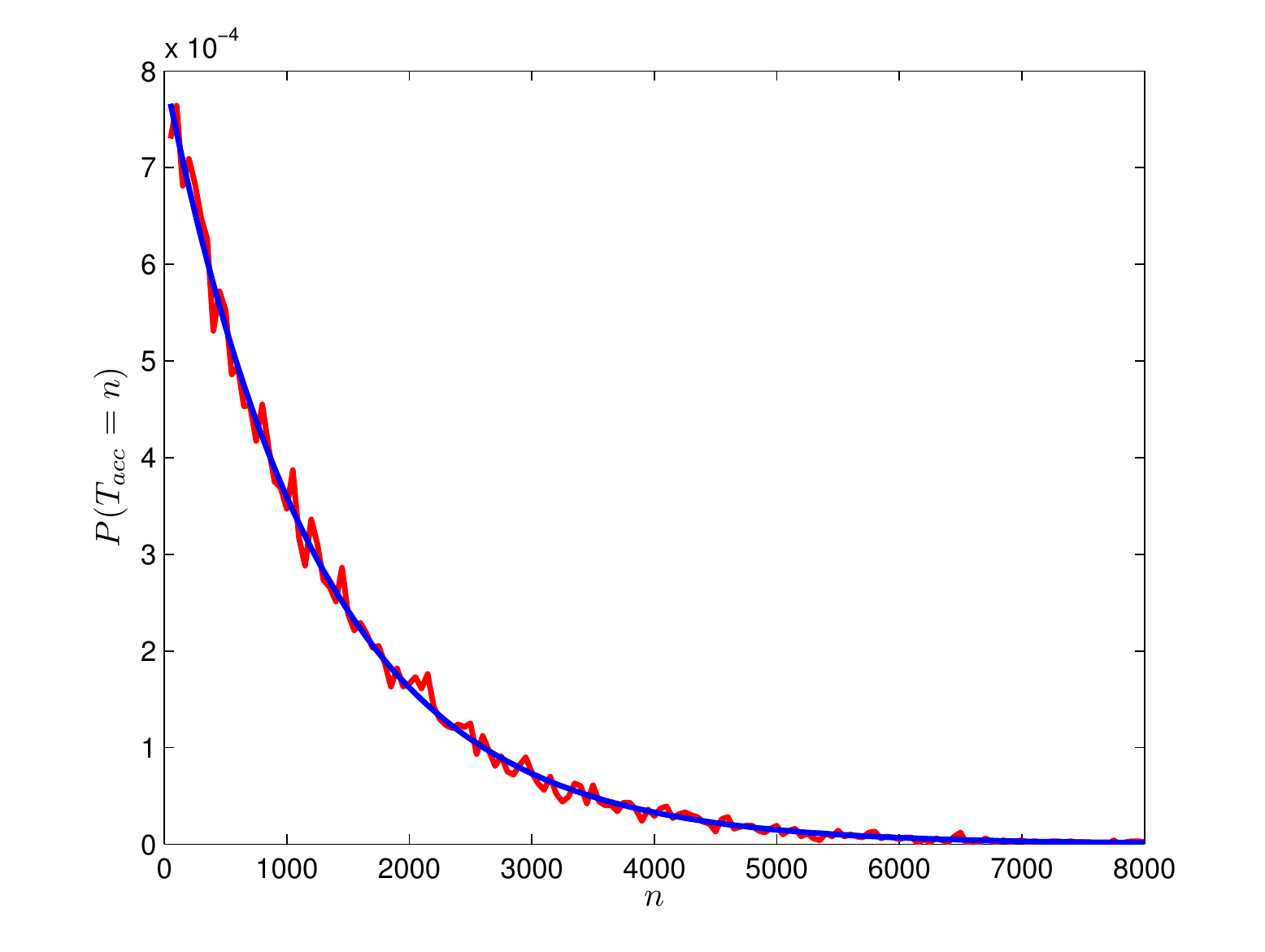}
  \end{center}
  \caption{Probability mass function of $\tacc$, estimated by $10^6$
    ParRep simulations with $N = 10$ replicas and $\tphase = \tcorr =
    25$, vs.  exact distribution of $\tau$ (smooth curve). }
  \label{fig1}
\end{figure}

\begin{figure}[htbp]
  \begin{center}
    \includegraphics[width=12cm]{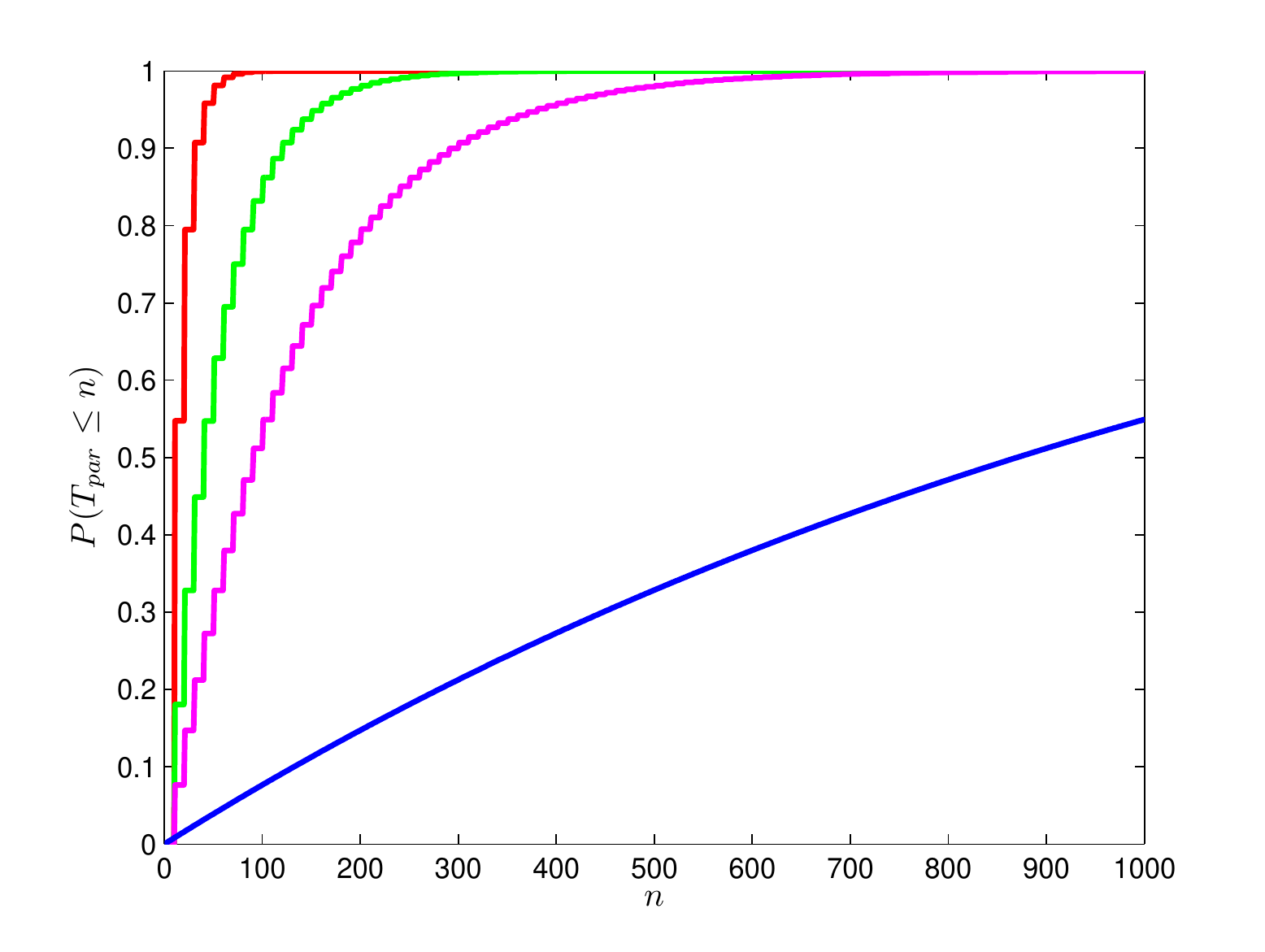}
  \end{center}
  \caption{Cumulative distribution function of parallel time required
    for ParRep sampling with $\tpoll = 10$ and, from top: $N = 100,
    25, 10$. The bottom curve is the (analytic) cumulative
    distribution function of $\tau$ (corresponding to $N=1$).}
  \label{fig2}
\end{figure}

\subsection{Discretized Diffusions}

Consider the overdamped Langevin stochastic process in ${\mathbb R}^d$,
\begin{equation}\label{ovdlang}
  d\tilde{X}_t = -\nabla V(\tilde{X}_t) dt + \sqrt{2\beta^{-1}} dW_t.
\end{equation}
The associated Euler-Maruyama discretization is
\begin{equation}
  \label{e:em_process}
  X_{n+1} = X_n - \nabla V(X_n) \Delta t + \sqrt{ 2\beta^{-1} \Delta t} \xi_n
\end{equation}
where $\xi_n \sim N(0,I)$ are $d$-dimensional i.i.d. random
variables. It is well-known~\cite{Kloeden} that $(X_n)_{n \ge 0}$ is then an
approximation of $(\tilde{X}_{n\Delta t})_{n \ge 0}$.

\subsubsection{Existence and uniqueness of the QSD}

We first show that the conditions in Assumption~\ref{assumption1} hold 
(see~\cite{delmoral} for a similar example in 1D):
\begin{proposition}\label{prop1}
  Assume $S \subset \R^d$ is bounded and $\nabla V$ is bounded on
  $S$. Then \eqref{e:em_process} satisfies Assumption
  \ref{assumption1}.
\end{proposition}
\begin{proof} First, for any $x \in S$, 
  \begin{equation}
    \begin{split}
      \mathbb{P}^x(X_1 \in S) = \E^x\bracket{1_{S}(X_1)}&= 
     (4\pi \beta^{-1}\Delta t)^{-d/2}\int_{\R^d}   1_S(y)    \exp\set{-\frac{\abs{y -x + \nabla V(x) \Delta t}^2}{4\beta^{-1} \Delta t} }dy\\
      & \geq |S|(4\pi \beta^{-1}\Delta t)^{-d/2} \min_{y\in S}\set{
        \exp\set{-\frac{\abs{y -x + \nabla V(x) \Delta
              t}^2}{4\beta^{-1} \Delta t} }}> 0.
    \end{split}
  \end{equation}
  Next, for any $x,y \in S$, 
  \begin{equation}
    \begin{split}
      \E^x\bracket{f(X_1)1_{\{\tau>1\}}} &= (4\pi \beta^{-1}\Delta t)^{-d/2}\int_S f(z)
      \exp\set{-\frac{\abs{z -x + \nabla V(x) \Delta
            t}^2}{4
          \beta^{-1} \Delta t} }dz\\
      & = (4\pi \beta^{-1}\Delta t)^{-d/2}\int_S f(z) \exp\set{-\frac{\abs{z -y + \nabla
            V(y) \Delta t}^2}{4
          \beta^{-1} \Delta t} }  \\
      &\quad \times \exp\set{-\frac{\abs{z -x + \nabla V(x) \Delta
            t}^2- \abs{z -y + \nabla V(y) \Delta t}^2}{4
          \beta^{-1} \Delta t}}dz\\
      &\geq C (4\pi \beta^{-1}\Delta t)^{-d/2}\int_S f(z) \exp\set{-\frac{\abs{z -y +
            \nabla V(y) \Delta t}^2}{4
          \beta^{-1} \Delta t} }  dz\\
      &\quad = C(4\pi \beta^{-1}\Delta t)^{-d/2} \E^y\bracket{f(X_1)1_{\{\tau>1\}}}
    \end{split}
  \end{equation}
  where
  \begin{equation*}
    C = \min_{x,y,z\in S} \exp\set{-\frac{\abs{z -x + \nabla V(x)
          \Delta t}^2- \abs{z -y + \nabla V(y)
          \Delta t}^2}{4
        \beta^{-1} \Delta t}}.
  \end{equation*}
  Since $S$ is bounded and terms in the brackets are bounded, $C>0$.
  In Assumption~\ref{assumption1} we can then take $m=1$ and $\delta =
  C(4\pi \beta^{-1}\Delta t)^{-d/2}$. 
\end{proof}

  Theorem~\ref{theorem1} ensures that $(X_n)_{n\ge 0}$ converges 
  to a unique QSD in $S$, with a precise error estimate in terms of
  the parameters $m$ and $\delta$ obtained in the proof 
  of Proposition~\ref{prop1} . This error estimate is certainly not
  sharp; better estimates can be obtained by studying the spectral
  properties of the Markov kernel. We refer to~\cite{Tony} for such convergence results in the continuous time case~\eqref{ovdlang}.

\subsubsection{Numerical example}\label{sec:diff1d}

Here we consider the 1D process
\begin{equation}
  \label{e:per1d}
  d\tilde{X}_t = - 2\pi \sin (\pi \tilde{X}_t) dt + \sqrt{2}dW_t, 
\end{equation}
discretized with $\Delta t = 10^{-2}$. We compute 
the first exit time from $S = (-1,1)$, starting at ${\tilde X}_0 = 1/2$. 
We use Algorithm~\ref{algorithm1} with 
$\tcorr = \tphase = 100$, corresponding to the
physical time scale $\tcorr \Delta t = \tphase \Delta t = 1$, 
and $N=1000$ replicas. 

Consider a direct implementation of the continuous time ParRep
algorithm into the time discretized
process. In that algorithm, the accelerated time is (in units of
physical time instead of time steps)
\begin{equation}
  \label{e:tacc_naive}
  \tacc^{\rm continuous} = N \tau^K  \Delta t,
\end{equation}
with $\tau^K$ the same as in Algorithm~\ref{algorithm1} above. As 
$\tacc^{\rm continuous}$ is by construction a multiple of $N \Delta t = 10$, 
a staircasing effect can be seen in the exit time distribution; see
Figure~\ref{f:per1d}.  This staggering worsens as the number of
replicas increases. In our Algorithm~\ref{algorithm1}, we use the accelerated
time formula (again in units of physical time)
\begin{equation*}
  \tacc^{\rm corrected} = \tacc \Delta t.
\end{equation*}
We find excellent agreement between
the serial data -- that is, the data obtained from direct numerical 
simulation -- and the data obtained from Algorithm~\ref{algorithm1}. 
See Figure~\ref{f:per1d}. (The agreement is perfect in the 
decorrelation step; see Figure~\ref{f:per1d_zoom}.) We comment further on this in the next section.

\begin{figure}
  \begin{center}
    \includegraphics[width=10cm]{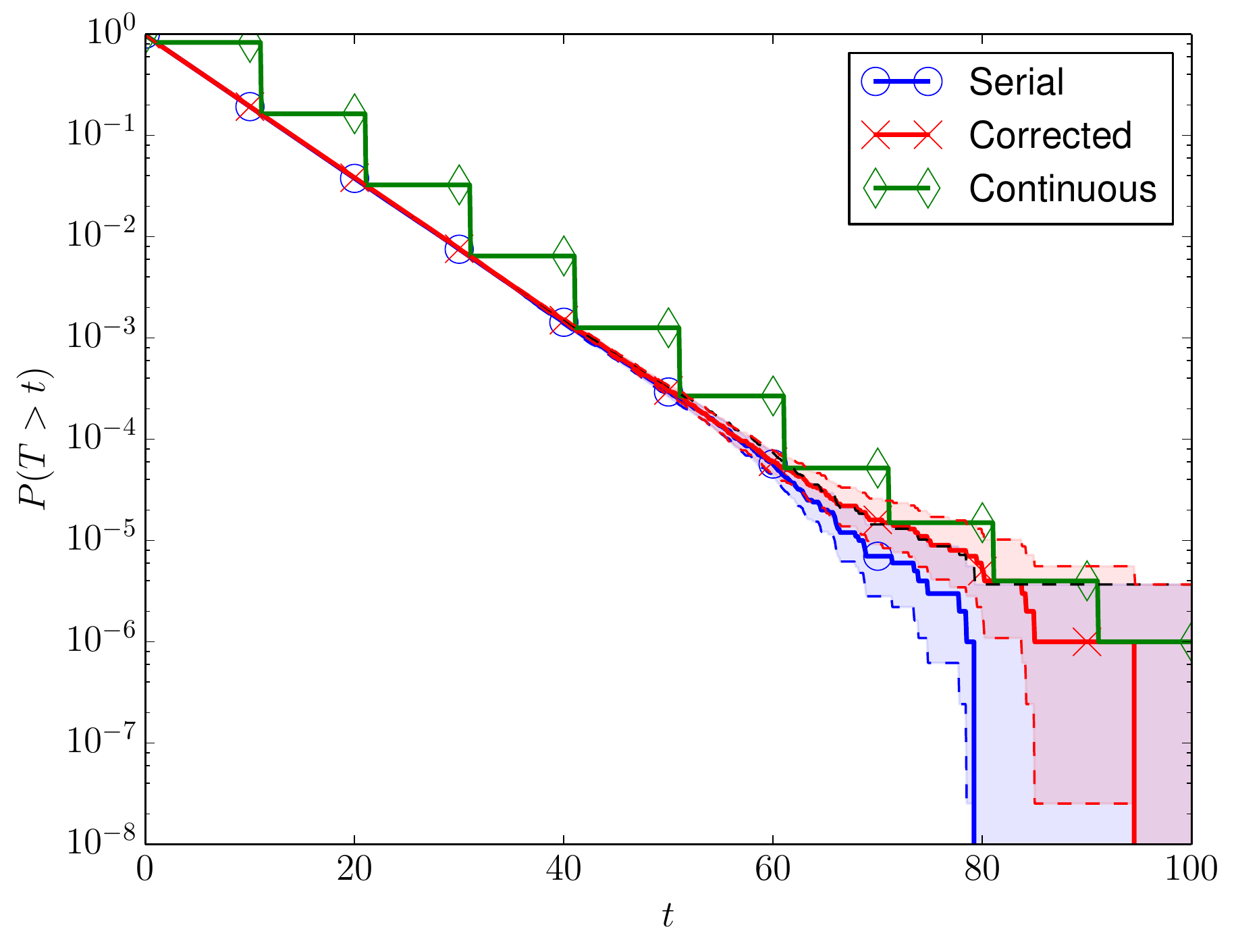}
  \end{center}
  \caption{Exit time distributions for the Euler-Maruyama
    discretization of \eqref{e:per1d}. Here $T$ represents 
    the first exit time from $S = (-1,1)$, starting at $1/2$. 
    There is excellent agreement between the serial, unaccelerated simulation
    data ($T = \tau^\nu \Delta t$) and our ParRep algorithm ($T =
    \tacc^{\rm corrected}$), while the original ParRep 
    formula ($T = \tacc^{\rm continuous}$) 
    deviates significantly. Dotted lines represent 
    95\% Clopper-Pearson confidence intervals obtained from 
    $10^6$ independent simulations; confidence interval widths increase in $t$ as 
    fewer samples are available.}
  \label{f:per1d}
\end{figure}

\begin{figure}
  \begin{center}
    \includegraphics[width=10cm]{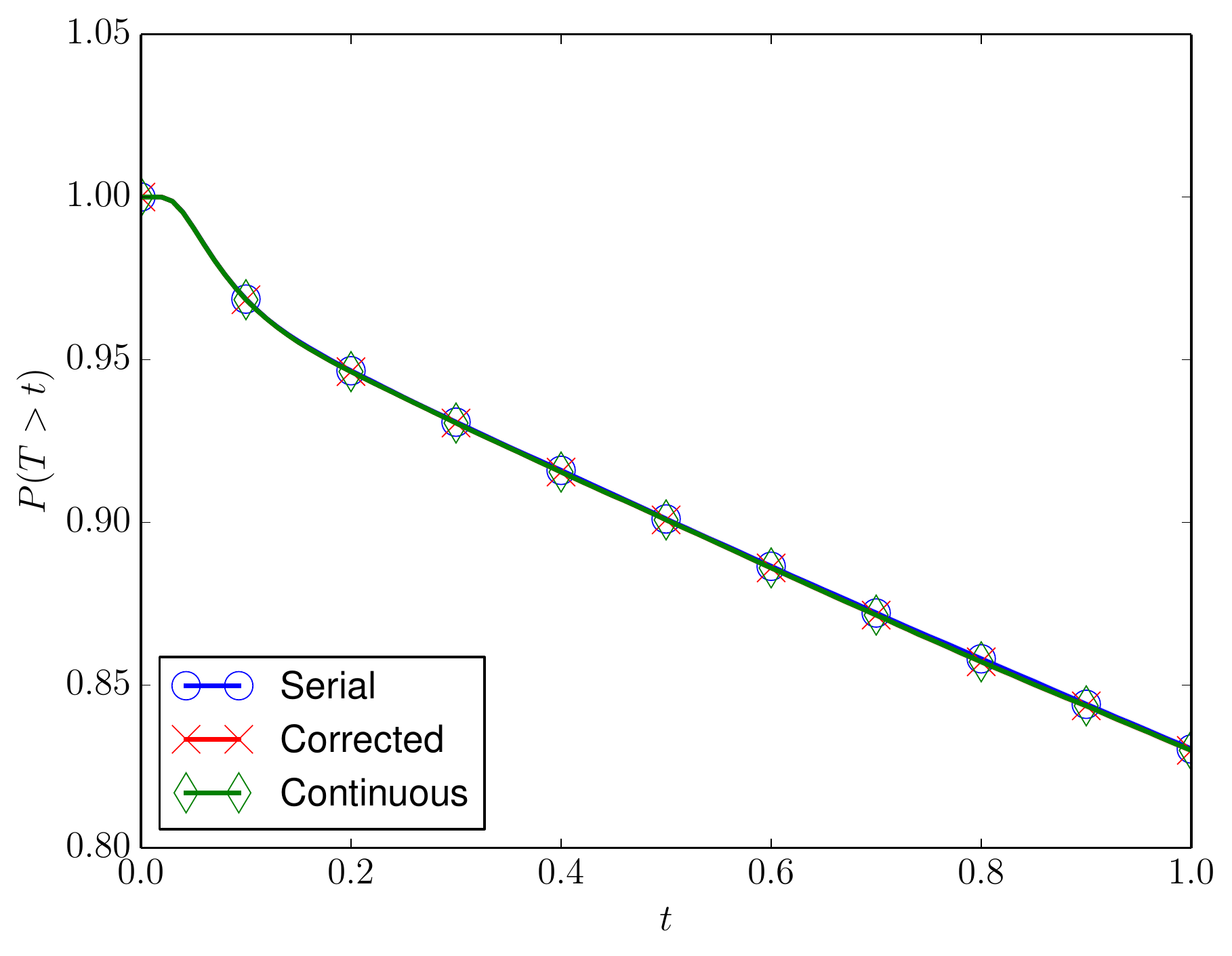}
  \end{center}
  \caption{A zoomed-in version of Figure~\ref{f:per1d}, 
highlighting the decorrelation step (recall $\tcorr \Delta t = 1$). 
    Serial
    simulation, our ParRep algorithm, and the original ParRep 
    algorithm all produce identical data.
    This comes from the fact that 
    serial and ParRep simulations are identical in law 
    during the decorrelation step. }
  \label{f:per1d_zoom}
\end{figure}

\subsubsection{Discussion}

In light of the discretization example, one may ask what kind of
errors were introduced in previous numerical studies which used ParRep
with \eqref{e:tacc_naive}. Taking $\tpoll =1$ for simplicity, we
calculate
\begin{equation*}
  \E\bracket{\abs{\tacc^{\rm corrected} - \tacc^{\rm continuous}}} = \E\bracket{\abs{ (N (\tau^K-1) + K)\Delta t - N 
      \tau^{K}\Delta t}}= \Delta t\,\E\bracket{\abs{N-K}}= \Delta t\sum_{k=1}^N (N-k) \PP(K=k).
\end{equation*}
Using calculations analogous to those used to study $\tacc$, it can be
shown that
\[
\PP(K=k) = \frac{(1-p)^{k-1}p}{1-(1-p)^N}.
\]
Therefore the error in the number of time steps per parallel step is
\begin{equation}
  \text{Absolute Error}  = \frac{N\Delta t}{1-(1-p)^N} - \frac{\Delta t}{p},\quad \text{Relative Error}  = \frac{pN}{1-(1-p)^N} -1.
\end{equation}
Consider the relative error, writing it as
\begin{equation*}
  pN \bracket{\frac{1}{1-r^N}- \frac{1}{(1-r)N}}, \text{ where } r=1-p.
\end{equation*}
We claim the quantity in the brackets,
\begin{equation}
  \label{e:relerr_prefactor}
  f(r,N):=  {\frac{1}{1-r^N}- \frac{1}{(1-r)N}} = \frac{r^N -N r + N-1}{Nr^{N+1} -
    N r^N - Nr +N},
\end{equation}
is bounded from above by one.  Indeed, for any $0<r<1$, we immediately
see that $f(r,N)$ is zero at $N=1$ and one as $N\to \infty$.  Let us
reason by contradiction and assume that $\sup_{r \in (0,1), N>0}
f(r,N) > 1$. Since $f$ is continuous in $N>0$ and $0<r<1$, there is
then a point $(r,N)$ such that $f(r,N)=1$; thus
\[
g_N(r) =0, \text{ where } g_N(r):= N r^{N+1} - (N+1)r^N +1.
\]
Note that $g_N(0) = 1$ and $g_N(1) = 0$ for all values of $N$.
Computing the derivative with respect to $r$, we observe
\[
g_N'(r) = -N(N+1)(1-r)r^{N-1}<0.
\]
Therefore, $g_N(r)$ is decreasing, from one at $r=0$ to zero at $r=1$,
in the interval $(0,1)$.  Hence, $g_N(r) =0$ has no 
solution, contradiction. We conclude that~\eqref{e:relerr_prefactor} is bounded from above by one.

Consequently, we are assured
\begin{equation}
  \text{Absolute Error} \leq N \Delta t, \quad \text{Relative Error} \leq pN.
\end{equation}
Thus, so long as $pN \ll 1$, the relative error using the 
accelerated time $\tacc^{\rm continuous}$ will be modest, especially for very metastable
states where $p\ll 1$. If also $N \Delta t \ll 1$, then the absolute
error will be small. 

The
above calculations are generic.  Though our discretized diffusion
example in Section~\ref{sec:diff1d} is a simple 1D problem, the errors 
displayed in Figure
\ref{f:per1d} are expected whenever the continuous time ParRep
rule~\eqref{e:tacc_naive} is used for a time discretized process. 
Though this error (as we showed above) will be small provided $Np \ll 1$ and
$N\Delta t \ll 1$, our Algorithm~\ref{algorithm1} 
has the advantage of being consistent for any $\Delta t$, 
including relatively large values of $N\Delta t$.

\section*{Acknowledgments}

We would like to thank the anonymous referees for their many
constructive remarks.
The work of {\sc D. Aristoff} and {\sc G. Simpson} was supported in
part by DOE Award DE-SC0002085.  {\sc G. Simpson} was also supported
by the NST PIRE Grant OISE-0967140. The work of {\sc T. Leli\`evre} is
supported by the European Research Council under the European Union's
Seventh Framework Programme (FP/2007-2013) / ERC Grant Agreement
number 614492.

\bibliographystyle{plain}

\end{document}